\documentclass[twoside]{amsart}
\usepackage{amssymb,verbatim}
\usepackage{amsmath}
\usepackage{amsthm}
\usepackage{graphicx}
\usepackage{color}
\usepackage[all]{xy}

\newtheorem{theorem}{Theorem}[section]

\newtheorem{Proposition}[theorem]{Proposition}
\newtheorem{Lemma}[theorem]{Lemma}

\newtheorem{Definition}[theorem]{Definition}
\newtheorem{example}[theorem]{Example}
\newtheorem{remark}[theorem]{Remark}

\newcommand{\cO}{{\mathcal O}}

\newcommand{\A}{\mathbb A}
\newcommand{\N}{\mathbb N}
\newcommand{\Q}{\mathbb Q}
\newcommand{\Z}{\mathbb Z}
\newcommand{\C}{\mathbb C}
\newcommand{\R}{\mathbb R}

\newcommand{\bb}[1]{\mathbb{#1}}
\newcommand{\m}[1]{\mathcal{#1}}

\newcommand{\s}{\sigma}

\newcommand{\f}{\varphi}
\newcommand{\ra}{\rightarrow}

\DeclareMathOperator{\Spec}{Spec}

\DeclareMathOperator{\Proj}{Proj}

\DeclareMathOperator{\codim}{codim}

\begin{document}


\title[Fano-Mori contractions of high length]{Fano-Mori contractions of high length on
projective varieties with terminal singularities} 
\author{Marco Andreatta, Luca Tasin}

\thanks{
We like to thank Edoardo Ballico, Cristiano Bocci, Paolo Cascini, Massimiliano Mella and Roberto Pignatelli for helpful conversations.  We thank the referee, whose suggestions remarkably improved the presentation of the paper.}

\address{Dipartimento di Matematica, Universit\`a di Trento, I-38123
Povo (TN)} 
\email{marco.andreatta@unitn.it, luca.tasin@unitn.it}


\begin{abstract} 
Let $X$ be a projective variety with $\Q$-factorial terminal singularities and let $L$ be an ample Cartier divisor on $X$.

We prove that if $f$ is a  birational contraction associated to an extremal ray $ R \subset \overline {NE(X)}$ such that $R.(K_X+(n-2)L)<0$ then $f$ is a weighted blow-up of a smooth point.

 We then classify divisorial contractions associated to extremal rays $R$ such that $R.(K_X+rL)<0$, where $r$ is a non-negative integer, and the fibres of $f$ have dimension less or equal to $r+1$.
\end{abstract}

\maketitle

\section{Introduction}

Let  $X$ be a normal projective  variety over $\mathbb C$ with $\Q$-factorial terminal singularities and let $n =\dim X$.

We consider the Kleiman-Mori cone of $X$, $ \overline {NE(X)}$, which is the closure of the cone generated by effective curves modulo numerical equivalence in $N_1(X, \R)$. By the famous Cone Theorem of Mori and Kawamata (see Theorem 3.7(1) in \cite{KollarMori}), the subcone $ \overline {NE(X)}_{K_X <0}:= \{C \in  \overline {NE(X)}: K_X.C < 0\}$ is locally polyhedral.
By the Contraction Theorem of Mori, Kawamata and Shokurov, to any extremal ray in $ \overline {NE(X)}_{K_X <0}$  one can  associate a map (contraction) $f:X\ra Z$ with connected fibres onto a normal projective variety $Z$, which contracts all curves in the ray (see Theorem 3.7 in \cite{KollarMori}). These contractions are the basic steps of the Minimal Model Program,  a program which takes an algebraic variety to a minimal model, i.e a variety on which the canonical class is nef (not negative on any curve).

Choose now a {\it polarization} of $X$, that is an ample Cartier divisor $L$ on $X$. Let $r$ be a non-negative rational number. One can  consider the subcone $\overline {NE(X)}_{(K_X+ rL) <0}=\{C \in  \overline {NE(X)}: (K_X+rL).C < 0\}$. This subcone contains just a finite number of extremal rays (see Theorem 3.7(2) in  \cite{KollarMori}). Since $K_X+(n+1)L$ is always nef (see for instance \cite{An11}, Theorem 5.1) we know that if $r \ge n+1$, then  the subcone is actually empty, whereas  if $r$ approaches $0$, then the subcone fills up all $ \overline {NE(X)}_{K_X <0}$. Note that the bigger $r$ is, the more negative $K_X$ is on the curves of the ray; in these cases the contractions should be simpler.

In a previous paper \cite{An11}, the first author described all extremal rays contained in the cone
$\overline {NE(X)}_{(K_X+ (n-2)L) <0}$: this  summarizes and generalizes a series of results,
by many authors, of the so called Adjunction Theory (see also \cite{BeltramettiSommese}).

\medskip
The first aim of the present paper is  to complete the description of the extremal rays $R = \R^+[C]$ contained in the cone
$\overline {NE(X)}_{(K_X+ (n-2)L) <0}$ whose associated contractions are birational. In \cite{An11} it was simply proved that 
the associated contraction contracts a divisor to a smooth point, here we prove that it is always a weighted blow-up,
the most straightforward possibility.

\begin{theorem} \label{n-2}
Let $X$ be a normal projective variety with $\Q$-factorial terminal singularities and let $L$ be an ample Cartier divisor on $X$. Let $R$ be an extremal ray in  $\overline {NE(X)}_{(K_X+ (n-2)L) <0}$ and let $f:X\ra Z$ be its associated contraction. Assume that  $f$ is  birational.  Then  $f$ is a weighted blow-up of a smooth point with weight $\s=(1,1,b,\ldots,b)$, where $b$ is a positive integer (see Definition \ref{weighted}).  
\end{theorem}

If $n=3$ the Theorem follows from the results in \cite{An11} and the main Theorem in \cite{Kaw}; our proof is however independent of \cite{Kaw}. 

The Theorem is proved by induction on $n$,  starting with case $n=2$, which is Castelnuovo contraction Theorem for $(-1)$-curves on a (smooth) surface. The inductive step is achieved by the usual procedure of Adjunction Theory (the so called Horizontal slicing described in Lemma \ref{horizontal}), once one has a {\it good divisor} in the linear system $|L|$. The existence of such a divisor, i.e. a divisor which has  terminal singularities, has been proved in the paper \cite{AW1}.

\medskip

\medskip
Secondly, we will include the above Theorem in a more general statement regarding a ray $R = \R^+[C]$  contained in the cone
$\overline {NE(X)}_{(K_X+ rL) <0}$, where $r$ is a non-negative integer, whose associated contraction is divisorial (i.e. it is birational and its exceptional locus is a divisor) with all fibres of dimension less or equal to $r+1$ (in simpler words the fibre dimension of the
contraction is not too big compared to the negativity of $R$). We prove the following Theorem, which is a generalization of Theorem 4.9 in  \cite{KollarMori92} (see also Theorem 3.2 in \cite{An1}).

\begin{theorem}
\label{main}
Let $X$ be a normal projective variety with $\Q$-factorial  terminal singularities and let $L$ be an ample Cartier divisor on $X$. Let $R$ be an extremal ray in  $\overline {NE(X)}_{(K_X+ rL) <0}$ where $r \in \mathbb N$ is a non-negative integer  and let $f:X\ra Z$ be its associated contraction. Assume that  $f$ is  divisorial and that all fibres have dimension less or equal to $r+1$.
Let $E$ be the exceptional locus of $f$ and set $ C:=f(E) \subset Z$.
\begin{enumerate}
\item Then $\codim_{Z}C=r+2$,  there is a closed subset $S \subset Z$ of codimension al least 3 such that $Z'=Z\backslash S$ and $C'=C\backslash S$ are smooth,  and $f':X'=X\backslash f^{-1}(S) \to Z'$ is a weighted blow-up along $C'$ with weight $\sigma=(1,1,b\ldots,b,0,\ldots,0)$, where the number of $b$'s is $r$ (see Definitions \ref{weighted} and \ref{global}). 

\item Let $\m I'$ be a $\sigma $-weighted ideal sheaf of degree $b$ for $Z' \subset X'$ (see Definition \ref{global}) and let $i: Z' \to Z$ be the inclusion; let also  $\m I := i_*(\m I')$ and $\m I^{(m)}$ be the $m$-th symbolic power of $\m I$  (see Definition \ref{symbolic}). Then $X =  \Proj \bigoplus_{m\ge 0} \m I^{(m)}.$
\end{enumerate}
\end{theorem}

\section{Contractions}

\label{contractions}

Our language is compatible with that of \cite{KollarMori}. In this section we recall some pertinent definitions and results which are used in the sequel. 

A {\sl contraction} is a surjective  morphism, $f:Y\ra T$, 
between normal varieties and  with connected fibres.

If $\dim Y > \dim T$ the contraction is said to be of fibre type,
otherwise it is birational. The set 
$ E = \{y \in Y : f \hbox{\ is\ not\ an\ isomorphism\ at\  }y\}$
is the exceptional locus of $f$. If $f$ is birational and $\dim E = \dim Y -1$,
then it is also called {\sl a divisorial contraction}.
For a contraction $f:Y\ra T$, a $\Q$-Cartier divisor $H$ such that
$H = \f^*A$ for some ample $\Q$-Cartier divisor on $T$ is called
a {\sl supporting divisor} for the contraction.

\medskip

A fundamental result in Mori Theory is the following.
\begin{theorem} [Contraction Theorem, {\cite[3.7(3)]{KollarMori}}]
Let $X$ be a variety with log-terminal singularities  and let $R  \subset  \overline {NE(X)}_{K_X<0}$ be an extremal ray. 

Then there exists a unique projective morphism
$\f :X \ra W$ onto a normal projective variety $W$ which is characterised
by the following properties:

\begin{itemize}
\item[i)] For any irreducible curve $C \subset X$, $\f(C)$ is  a point if and
only if $[C] \in R$,
\item[ii)] $\f$ has connected fibres.

\end{itemize}
\label{F-Mcontr}
\end{theorem}

We call such contraction  the Fano-Mori (F-M) contraction associated to $R$; note that $-K_X$ is $\f$-ample. 

In studying the contraction associated to an extremal ray $R$ it makes sense to fix a
fibre and understand the contraction locally, i.e.
restricting to an affine neighbourhood of the image of the fixed fibre. The complete contraction can then be obtained by gluing
different local descriptions.

For this we use the local set-up developed by Andreatta--Wi\'sniewski, see
\cite{AW1} for the details. Roughly summarizing, let $f:Y\ra T$ be a contraction; fix a fibre $F$ of $f$ and take an open affine
$Z\subset T$ such that $f(F)\in Z$ and $\dim f^{-1}(z)\leq \dim F$, for $z\in Z$.
If $X=f^{-1}Z$,
then $f:X\ra Z$ is called a {\it local contraction around
$F$}. If there is no need to specify fixed fibres, then we will simply
say that $f:X\ra Z$ is a local contraction.
Note that in this setting we have  $H^0(X, K_X+ \tau L)=H^0(X, \cO _X)$ and  $Z=\Spec(H^0(X,\cO_X))$. One advantage of the local set-up is that we may always assume that $|L|$ is not empty.

\medskip

Let $L$ be an ample Cartier divisor and let $R \subset \overline {NE(X)}_{(K_X+ rL) <0}$, where $r$ is a non-negative rational number.
 Let $\f:X \ra Z$ be the local contraction around a fibre $F$ associated to $R$.

We can define the 
{\it nef value} of the pair $(X,L)$ as 
$$\tau(X,L):=\hbox{inf} \{t \in \R : K_X + tL \hbox{ is $\f$-nef}\}.$$
By the rationality theorem of Kawamata (Theorem 3.5 in  \cite{KollarMori}), $\tau(X,L)$ is a rational non-negative number. 

The assumption $R \subset \overline {NE(X)}_{(K_X+ r L) <0}$ implies that $\tau > r$.

Moreover note that the $\Q$-Cartier divisor $K_X+ \tau L$ is a supporting divisor of $\f$.
\medskip

We  have a  lower bound for the dimension of a fibre of a F-M contraction.

\begin{theorem}[{\cite[Theorem 2.1]{An1}}]
Let $\f:X\ra Z$ be a local contraction  supported by $K_X+\tau L$. If $\f$ is birational, then $\dim F \geq \tau$.
\label{dimF}
\end{theorem}

The following base point free theorem is the main technical tool of the paper.

\begin{theorem}[{\cite{AW1}}]\label{bpf}
Let $\f:X\ra Z$ be a local contraction  supported by $K_X+\tau L$. If $\f$ is birational and  $\dim F \leq \tau +1$, then $L$ is $\f$-base point free.
\end{theorem}

By abuse of notation, in the setting of Theorem \ref{bpf}, we will sometime say that  $L$ is  base point free on $X$.

\medskip

To apply inductive arguments we will need the following two lemmata, which are consequences of Bertini's theorem. 

\begin{Lemma}[Vertical slicing, {\cite[Lemma 2.5]{AW1}}] Let $\f:X\ra Z$ be a birational local contraction supported by $K_X+\tau L$, where $X$ has terminal singularities and $\tau \ge 0$. Let $X'$ be the divisor defined by a general global function $h \in H^0(X, \cO _X)=H^0(X, K_X+\tau L)$. Then $X'$ has terminal singularities and $\f_{|X'}: X' \to Z'$ is a local contraction supported by $K_{X'}+\tau L_{|X'}$. 
\label{vertical}
\end{Lemma}

\begin{Lemma}[Horizontal slicing, {\cite[Lemma 2.6]{AW1}}]
Let $\f:X\ra Z$ be a birational local contraction  supported by $K_X+\tau L$, where $X$ has terminal singularities and $\tau \ge 1$. Let $X' \in |L|$ be a general divisor and let $\f'=\f_{|X'}: X' \to Z'$.
\begin{itemize}
\item[i)] Outside of the base locus of $|L|$, $X'$ has terminal singularities.
\item[ii)] If $X'$ is normal, then $\f'$ is a local contraction supported by $K_{X'} + (\tau-1)L_{|X'}$.
\label{horizontal}
\end{itemize}
\end{Lemma}

\section{Weighted blow-ups}
\label{s_weighted}
In this section we consider weighted blow-ups along  smooth subvarieties; they are a generalization of weighted blow-ups of points as defined, for example, in Section 10 of \cite{KollarMori} or in Section 3 of \cite{Hayakawa}.

\medskip
Let $\sigma=(a_1,\ldots,a_k,0,\ldots,0) \in \mathbb N^n$ such that $a_i >0$ and $\gcd(a_1,\ldots,a_k)=1$. Let $M=\mathrm{lcm}(a_1,\ldots,a_k)$. 
We denote by  $ \bb P(a_1, \ldots,a_k)$ the weighted projective space with weight $(a_1,\ldots,a_k)$, that is the quotient
$
\bb P(a_1,\ldots,a_k)=(\mathbb A^{k} \backslash \{0\})/\C^*
$
where the action on $\mathbb A^k$ (with coordinates $y_1,\ldots,y_k$) is given by
$
\lambda \cdot (y_1,\ldots,y_k)= (\lambda^{a_1}y_1,\ldots,\lambda^{a_k}y_k)  \quad \mbox{ for} \lambda \in \C^*.
$

\smallskip

\begin{Definition}
\label{weighted}
Let $X=\mathbb A^n = \Spec  \C[x_1, \dots, x_n]$  and $Z=\{x_1=\ldots=x_k=0 \} \subset X$.  
Consider the rational map
$$
\f: \bb A^n \to \bb P(a_1, \ldots,a_k)
$$
given by $(x_1,\ldots,x_n) \mapsto (x_1^{a_1}:\ldots:x_k^{a_k})$.

The \textit{weighted blow-up} of $X$ along $Z$ with weight $\sigma$  is defined as the closure $\overline X$ in 
$\bb A^n \times \bb P(a_1, \ldots,a_k)$ of the graph of $\f$, together with the morphism $\pi: \overline X \to X$ given by the projection on the first factor.
\end{Definition}

The map $\pi$ is birational and contracts an exceptional irreducible divisor $E$ to $Z$. Moreover for any point $z \in Z$ we have $\pi^{-1}(z)=\mathbb P(a_1,\ldots,a_k)$.

\medskip

We now describe an affine covering for $\bar X$, where each affine set has an {\it orbifold structure}; this covering comes out naturally from the Toric construction of the weighted blow-up, as done in  \cite{KollarMori}. 

Given integers  $b_1,\ldots,b_n$ we define an action of $\Z_r$, the cyclic group of order $r$,  on $\A^n= \Spec\C[x_1,\ldots,x_n]$ by $\tau \cdot x_i = \varepsilon ^{b_i}x_i$, where $\tau$ is a generator of $\Z_r$ and $\varepsilon$ is a primitive $r^{th}$-root of unity. The quotient is denoted by $\A^n / Z_r(b_1,\ldots,b_n)$ and it is called a cyclic quotient singularity of type $1/r(b_1,\ldots,b_n)$.
 
Let $(y_1: \ldots :y_k)$ be homogeneous coordinates on $\bb P(a_1, \ldots,a_k)$. For any $i=1,\ldots, k$ consider the open  subset $U_i=\overline X \cap \{y_i \ne 0  \} \subset \bb A^n \times \bb P(a_1, \ldots,a_k)$.

One can see that
\begin{align*}
U_i &\cong  \Spec  \C[\bar x_1, \ldots, \bar x_n]/\bb Z_{a_i}(-a_1,\ldots, \overset{i\mbox{\scriptsize-th}}{1},\ldots,-a_k,0,\ldots,0) \\
 &\cong  \left( \Spec  \C[\bar x_1, \ldots, \bar x_k]/\bb Z_{a_i}(-a_1,\ldots,\overset{i\mbox{\scriptsize-th}}{1},\ldots,-a_k) \right)\times \bb A^{n-k}.
\end{align*}

and 
$$
\pi_{|U_i} : U_i \ni (\bar x_1, \ldots, \bar x_n) \mapsto (\bar x_1 \bar x_i^{a_1}, \ldots,\overset{i\mbox{\scriptsize-th}}{\bar x_i^{a_i}}, \ldots, \bar x_k \bar x_i^{a_k}, \bar x_{k+1}, \ldots, \bar x_n) \in X.
$$

In the affine set $U_i$, $E$ is defined by $\{\bar x_i=0  \}/\bb Z_{a_i}(-a_1,\ldots,\,1,\ldots,-a_k,0,\ldots,0)$; hence $ME$ is a Cartier divisor on $\overline X$.

\bigskip

We define the function 
$$
\sigma\textrm{-wt}: \bb C[x_1,\ldots,x_n] \to \bb \N
$$
as it follows. For a monomial $T=x_1^{s_1}\ldots x_n^{s_n}$ we set $\sigma\textrm{-wt}(T):=\sum_{i=1}^k s_i a_i$. For a polynomial $f= \sum_I \alpha_I T_I$, where $\alpha_I \in \bb C$ and $T_I$ are monomials, we set 
$$
\sigma\textrm{-wt}:=\min \{\sigma\textrm{-wt}(T_I) : \alpha_I \ne 0 \}.
$$

\begin{Definition}
\label{weightedId}
Let $\sigma=(a_1,\ldots,a_k,0,\ldots,0) \in \mathbb N^n$ such that $a_i >0$ and $\gcd(a_1,\ldots,a_k)=1$. For any $d \in \N$ we define the $\sigma$-weighted ideal of degree $d$ as 
$$ 
I_{\s,d}= \{g \in \C[x_1,\ldots,x_n] : \sigma\textrm{-wt}(g)\ge d  \}= (x_1^{s_1}\cdots x_n^{s_n} : \sum_{j=1}^k s_ja_j \ge d).
$$
\end{Definition}

\begin{Lemma} \label{push-forward}
Let $\pi: \overline{X} \to X$  be the weighted blow-up of $X=\mathbb A^n$  along $Z=\{x_1=\ldots=x_k=0 \}$ with weight  $\sigma=(a_1,\ldots,a_k,0,\ldots,0)$. 
Then
$$
\pi_* \m O_X(-dE)= I_{\s,d}.
$$
\end{Lemma}
Therefore 
$$
\overline X= \Proj \bigoplus_{d \ge 0} I_{\sigma,d}.
$$
\begin{proof}
Let $\{U_i\}$ be the standard affine covering of $\overline X$.

As the exceptional divisor $E$ is effective, we have that $J:=\pi_* \m O_X(-dE) \subset \C[x_1,\ldots,x_n]$ is an ideal.  A polynomial $g(x_1,\ldots,x_n)$ is an element of $J$ if and only if for any $1\le i \le k$ we have that $g(\bar x_1,\ldots,\bar x_n) \in \Gamma(U_i, \m O_X(-dE))$.  Since $E$  is defined by $\{\bar x_i =0\}/\Z_{a_i}$ on the affine subset $U_i$, we have  that $\bar x_i^d$ divides $g(\bar x_1,\ldots,\bar x_n)$ if and only if $\sigma\textrm{-wt}(g) \ge d$, and the lemma follows.
\end{proof}

In the case of the blow-up with weight $(1,1,b,\ldots,b,0\ldots,0)$ we can say something more.

\begin{Proposition}\label{projective-normality}
Let $\pi: \overline{X} \to X$  be the weighted blow-up of $X=\mathbb A^n$  along $Z=\{x_1=\ldots=x_k=0 \}$ with weight  $\sigma=(1,1,b, \ldots,b,0,\ldots,0)$ (where the number of $b$'s is $k-2$). Then
$$
\pi_* \m O_X(-dbE)=\pi_* \m O_X(-bE)^d 
$$
and
$$
\overline{X}= \Proj \bigoplus_{d\ge 0} I_{\s,b}^d.
$$
\end{Proposition}

\begin{proof}

It suffices to prove that for every integer $d \ge 1$ the natural map
$$
\pi_* \m O_X(-(d-1)bE) \bigotimes \pi_* \m O_X(-bE) \to \pi_* \m O_X(-dbE)
$$
is surjective. For $d=1$ there is nothing to prove, so we assume $d \ge 2$.

\smallskip

Let $g=x_1^{s_1}\cdots x_n^{s_n} \in \pi_* \m O_X(-dbE)=(x_1^{s_1}\cdots x_n^{s_n} : s_1+s_2+\sum_{i=3}^k s_ib \ge db)$. We claim that there exists   $h =x_1^{t_1}\cdots x_n^{t_n} \in \pi_* \m O_X(-bE)$   such that $t_1+t_2+\sum_{i=3}^k t_ib=b$ and $t_i \le s_i$ for all $1\le i \le n$. 

In fact, if there is $j \in \{3,\ldots,k\}$ such that $s_j \ne 0$, then just set $h=x_j$. If $s_j=0$ for $j=3,\ldots,k$, then either $s_1 \ge b$ or $s_2 \ge b$ and the claim follows setting $h=x_1^b$ or $h=x_2^b$. 

Let $k=g\cdot h^{-1}$, then $k \in \pi_* \m O_X(-(d-1)bE)$ and $g=k \cdot h$.

The second equality is a consequence of the first equality and Lemma \ref{push-forward}.
\end{proof}

We remark that the previous Proposition does not hold for any weight $\sigma$ as the following example shows.  Nevertheless, since the algebra
$$
\bigoplus_{d\ge 0} I_{\s,b}^d
$$
is finitely generated, there is always a positive integer $L$ such that
$$
I_{\s,L}^d=I_{\s,dL}. 
$$
for any $d \in \N$.

\begin{example}
Let $Z=\{x_1=x_2=x_3=0 \} \subset X= \mathbb A^n$ and $\sigma=(10,14,35,0,\ldots,0)$. Let $\pi: \overline X \to X$ be the weighted blow-up of $X$ along $Z$ with weight $\sigma$. Consider  
$$
g=x_1^5x_2^4x_3 \in \pi_* \m O_X(-2ME),
$$
where $M=\textrm{lcm}(2,5,7)=70$; note that $\sigma\textrm{-wt}(g)=141$. It is easy to check that there is no triple $(t_1,t_2,t_3) \in \mathbb N^3$ such that $10t_1+14t_2+35t_1=70$ and $t_1 \le 5$, $t_2 \le 4$, $t_3 \le 1$ and hence
$$
\pi_* \m O_X(-140E) \ne \pi_* \m O_X(-70E)^2.
$$

\end{example}

\medskip
\begin{Lemma} \label{reidcrit}
Let $\pi: \overline{X} \to X$  be the weighted blow-up of $X=\mathbb A^n$  along $Z=\{x_1=\ldots=x_k=0 \}$ with weight  $\sigma=(1,1,b, \ldots,b,0,\ldots,0)$ (where the number of $b$'s is $k-2$).
Then $ \overline X$ has  terminal singularities.
\end{Lemma}

\begin{proof} By the above description of the open subsets $U_i \subset \overline X$ we must show that a cyclic quotient singularity of type $\frac{1}{b}(b-1,b-1,1,0,\ldots,0)$ is terminal. This follows immediately by \cite[Thm. 4.11]{Re87}. 
\end{proof}

 \bigskip
We now define the symbolic powers of an ideal and check that $\sigma$-weighted ideals behave well with respect to symbolic powers. 

\begin{Definition}
\label{symbolic}
Given an ideal $I \subset R$ in a Noetherian ring $R$ and $t \in \mathbb N$, the $t$-th symbolic power $I^{(t)}$ of $I$  is defined as the restriction of $I^tR_S$ to $R$, where $S$ is the complement of the union of the minimal associated primes of $I$ and $R_S$ is the localization of $R$ at the multiplicative system $S$.  
\end{Definition}

If $I$ is a prime ideal, then  the definition of symbolic power is for instance given in \cite[Definition 9.3.4]{Laz} or in \cite[Exercise 4.13]{AM69}. Note that, by definition, $I^t \subset I^{(t)}$; in general the inclusion might be strict.

\begin{Lemma}\label{symbolic-powers}
Let $\sigma=(a_1,\ldots,a_k,0 \ldots,0) \in \N^n$  such that $a_i >0$ and  $\gcd(a_1,\ldots,a_k)=1$, and let $L$ be a positive integer such that
$$
I_{\s,L}^d=I_{\s,dL}.
$$
for any $d \in \N$, where $I_{\s,d}$ is the $\s$-weighted of degree $d$. Set $I=I_{\s,L}$. 

Then for any $t \in \N$ we have $I^t=I^{(t)}$.
\end{Lemma}

\begin{proof}
We will use the fact that if $f,g \in  R$ then $\sigma\textrm{-wt}(fg)=\sigma\textrm{-wt}(f)+\sigma\textrm{-wt}(g)$.

We first show that $I$ is primary: if $fg \in I$ then $\sigma\textrm{-wt}(f) \ge 1$ or $\sigma\textrm{-wt}(g) \ge 1$ and hence $f^m \in I$ or $g^m \in I$ for $m$ big enough. Then the only prime associate to $I$ is its radical ideal $r(I)$, which is $r(I)=(x_1,\ldots,x_k)$ since there is always a power of $x_i$ in $I$ for $1\le i\le k$. 

Let now $S=R \backslash r(I)$.  By Proposition 3.11 in \cite{AM69} we have $I^{(t)}=\bigcup_{s \in S}(I^t:s)$. Using  the fact that  $\sigma\textrm{-wt}(s)=0$, we have that for any $s \in S$ 
$$
(I^t : s)=\{g \in R : \sigma\textrm{-wt}(gs)\ge tL \}=\{g \in R : \sigma\textrm{-wt}(g)\ge tL \}=I^t.
$$

\end{proof}

\bigskip

The  definition of weighted blow-up in  \ref{weighted} depends on the local coordinates chosen. To construct a global weighted blow-up along a subvariety $Z$ of a complete variety $X$ one needs to patch together weighted blow-ups defined on a covering of $X$, in such a way that the local coordinates preserve the weight. We propose the following.

\begin{Definition}\label{global}
Let  $X$ be a smooth variety and $Z$ a smooth subvariety of codimension $k$  and let $\sigma=(a_1,\ldots,a_k,0\ldots,0) \in \mathbb N^n$  such that $a_i >0$ and $\gcd(a_1,\ldots,a_k)=1$. 
Let $\m I_{\s,d}$ be ideal sheaves on $X$ such that there is a covering $\{U_i \cong \C^n\}_{i \in I}$ on $X$ 
so that for any $i \in I$  there are local coordinates $x_1,\ldots,x_n$ on $U_i$ for which 
$Z\cap U_i=\{x_1=\ldots=x_k=0\}$  and $\Gamma(U,\m I_{\s,d})=\{g \in \C[x_1,\ldots,x_n] : \sigma\textrm{-wt}(g)\ge d  \}$.
A \textit{weighted blow-up} of $X$ along $Z$ with weight $\sigma$ is the projectivization
$$
\pi: \overline X= \Proj \bigoplus_{d\ge 0} \m I_{\s,d}  \to X.
$$
We call $\m I_{\s,d}$ a \textit{$\sigma$-weighted ideal sheaf} of degree $d$ for $Z$ in $X$.

\end{Definition}

\begin{remark}
Let $Z \subset X$ be a smooth subvariety of codimension $k$; let also  $\sigma=(a_1,\ldots,a_k,0,\ldots,0)$ be a weight. The question about the existence of a \textit{$\sigma$-weighted ideal} for $Z$ in $X$, and therefore of a weighted blow-up of $X$ along $Z$ with weight $\sigma$, is not clear. In general  it seems a difficult problem to find sufficient conditions for a positive answer. 

However, if $Z \subset \bb P^n$ is a complete intersection, then for any weight $\sigma=(a_1,\ldots,a_k,0,\ldots,0)$ there exists a weighted blow-up along $Z$ with weight $\sigma$. In fact, let $F_1,\ldots, F_k$ be a regular sequence of homogeneous polynomials generating the ideal $I_Z\subset R=\C[x_0,\ldots,x_n]$ of $Z$. Then we define a function,  $\sigma\textrm{-wt}$, on $\C[x_0,\ldots,x_n]$ as follows.  If $g \not\in I_Z$ then just set $\sigma\textrm{-wt}(g)=0$. If $g \in I_Z$, then write $g= \sum_{\beta \in \N^{k}} h_\beta F_1^{\beta_1}\cdots F_k^{\beta_k} $, where $\beta=(\beta_1,\ldots,\beta_k)\in \N^k$ and $h_\beta \in R \backslash I_Z$, and set
$$
\sigma\textrm{-wt}(g):=\min_{\beta} \left\{\sum_{i=1}^{k} a_i\beta_i : h_\beta \ne 0\right\}.
$$

Note that the function $\sigma\textrm{-wt}$ is a well defined because $F_1, \ldots, F_k$ is a regular sequence. 

For any  $d \in \N$, we define the $\s$-weighted ideal sheaf  $\m I_{\s,d} \subset \m O_X$ setting,  for any open subset $U \subset \bb P^n$,
$$
\Gamma(U, \m I_{\s,d})=\left\{\frac{f}{g} : f,g \in \C[x_0,\ldots,x_n] \mbox{ are homogeneous, $g_{| U} \not\equiv 0$ and }\sigma\textrm{-wt}(f)-\sigma\textrm{-wt}(g) \ge d  \right\}.
$$

\end{remark}

\section{Proofs}

\begin{proof}[Proof of Theorem \ref{n-2}]

Let $F_1$ be a non trivial fibre of $f$. We pass to a local set-up, i.e. we assume that $f: X \to Z$  is a local F-M contraction around $F_1$ supported by $K_X+\tau L$, where $\tau$ is the nef value of the pair $(X,L)$, a positive rational number greater than $n-2$. Let $F$ be a component of $F_1$. By  Theorem \ref{dimF}, $\dim F \ge \tau > n-2$; this means that $\dim F=n-1$ and hence the exceptional locus of $f$ has codimension 1. Since the exceptional locus of a divisorial contraction is irreducible we conclude that $F$ is the exceptional divisor, and $f$ is the contraction of $F$ to a point $p \in Z$. 
\medskip

Since $f$ is a $K_X$-negative contraction, $Z$ is terminal and $\Q$-factorial (see \cite[Corollary 3.43]{KollarMori}).
Proposition 3.6 of \cite{An11} says that  $p$ is smooth. For the reader's convenience we repeat that proof. More precisely, by induction on $n \ge 2$, we prove that $p$ is smooth in $Z$ and that 
\begin{align}\label{formulas}
K_X= f^*K_Z + ((n-2)b+1)F \ \mbox{ and } \ \tau= n-2 + \frac{1}{b}
\end{align}
for a positive integer $b$ such that $f^* f_* L= L+ bF$.

If $n=2$, $X$ is smooth and in this case we can apply  Castelnuovo's  theorem, which says that $f$ is the contraction of a $(-1)$-curve $F$ to a smooth surface $Z$, therefore $K_X= f^*K_Z +F$. Note that $L_1:=f_*L$ is a Cartier divisor on $Z$ and there is a positive integer $b$ such that $L=f^*L_1 -bF$. From $0=(K_X + \tau L).F=(K_X-\tau b F).F= -1 + \tau b$, we get $\tau=1/b$.

Let $n\ge 3$ and pick a general member $X' \in |L|$: by Theorem \ref{bpf} and Bertini's theorem  it has terminal $\Q$-factorial singularities. Consider the restricted morphism $f':= f_{|X'}: X' \to Z'$; by Lemma \ref{horizontal} it is a divisorial contraction supported by $K_{X'}+(\tau-1)L_{|X'}$. By inductive assumption, $p \in Z'$ is smooth; by \cite[Lemma 1.7]{Me97} we conclude that $p \in  Z$ is smooth, $L_1:= f_*L$ is Cartier and  $L=f^*L_1 -bF$ for a positive integer $b$. Denoting by $F'$ the exceptional divisor of $f'$, by induction we have 
$$
K_{X'}= f^*K_{Z'} + ((n-3)b+1)F' \ \mbox{ and } \ \tau-1= n-3 + \frac{1}{b},
$$
from which \eqref{formulas} follows.

\medskip

Let $X' \in |L|$ be again a general element  and $f':= f_{|X'}: X' \to Z'$ be the restricted morphism. Since $Z$ and $Z' =  f_*X'$ are smooth at $p$ and $f$ is local, we may assume that $Z= \mathbb A^n=\Spec \mathbb C[x_1,\ldots,x_n]$, where $x_1,\ldots,x_n$ are local coordinates for $p$, and that $f_*X'=\{x_n=0\}$.

Note that $\m O_X(-bF)$ is  $f$-ample  and that the map $f$ is proper; so we have that 
$$
X= \Proj ( \oplus_{d \ge 0}  f_*\m O_X(-dbF)).
$$ 

By Lemma \ref{push-forward}, $X$ will be the weighted blow-up we are looking for if 
$$
f_*\m O_X(-dbF)=I_{\s,d}=(x_1^{s_1}\cdots x_n^{s_n}  :  s_1 + s_2 + \sum_{j=3}^n b s_j \ge db).
$$ 

The proof of this is by double induction on $n$ and $d$, starting with $n=2$ and $d=0$.

Consider the exact sequence
$$
0 \to \m O_X(-L-dbF) \to \m O_{X} (-dbF)\to \m O_{X'}(-dbF) \to 0.
$$

Note that
$$
-L-dbF \sim_f -(d-1)bF \sim_f K_X + (n-3 + d+ \frac{1}{b})L. 
$$
Hence, pushing down to $Z$ the above exact sequence and applying the relative Kawamata-Viehweg Vanishing, we have

\begin{align}\label{sequence1}
0 \to  f_* \m O_X(-(d-1)bF) \stackrel{\cdot x_n}\rightarrow f_*\m O_{X} (-dbF)\to f_*\m O_{X'}(-dbF) \to 0.
\end{align}

By induction on $n$, we can assume that
$$
f_*\m O_{X'}(-dbF)=(x_1^{s_1}\cdots x_{n-1}^{s_{n-1}}  :  s_1 + s_2 + \sum_{j=3}^{n-1}bs_j \ge db)
$$
where $s_j \in \bb N$. The case $n=2$ follows from Castelnuovo's theorem. By induction on $d$, we can also assume that
$$
f_*\m O_{X}(-(d-1)bF)=(x_1^{s_1}\cdots x_n^{s_n}  :  s_1 + s_2  + \sum_{j=3}^n bs_j \ge (d-1)b),
$$
the case $d=0$ being trivial.

Let $g=x_1^{s_1}\cdots x_n^{s_n} \in f_*\m O_{X}(-dbF)$ be a monomial.

If $s_n \ge 1$ then $g$, looking at the sequence \eqref{sequence1}, comes from $f_*\m O_{X}(-(d-1)bF)$ by the multiplication by $x_{n}$; therefore 
$$
s_1 + s_2 + \sum_{j=3}^{n-1} s_j b +s_n b \ge  (d-1)b + s_n b \ge db.
$$

If $s_n=0$, then  $g \in f_*\m O_{X'}(-dbF)$ and so 
$$
s_1 + s_2 + \sum_{j=3}^{n} s_j b= s_1 + s_2 + \sum_{j=3}^{n-1} s_j b \ge db.
$$
The non-monomial case follows immediately.

\end{proof}

\bigskip

To prove  Theorem \ref{main} we need some preliminary lemmata. 

The following is a local version of Theorem \ref{main} around a general fibre.

\begin{Lemma} \label{local}
Let $f: X \to Z$ be a local contraction supported by $K_X + \tau L$, where $X$ is terminal $\mathbb Q$-factorial and  $L$ is an $f$-ample Cartier divisor. Assume that $f$ is divisorial and let $E$ be the exceptional divisor. Set $C=F(E) \subset Z$. Assume also that there exists a positive integer $r$ such that $\tau >r$ and that any  fibre of $f$ has dimension less or equal to $r+1$. Let $F$ be a general nontrivial fibre.
Then $f(F)$ is a smooth point and, possibly shrinking $Z$ to a smaller affine neighbourhood of $f(F)$ and choosing appropriate coordinates, we have that $C=\{x_1= ... =x_{ r +2}=0  \} \subset \C^n = Z$ and  $f$ is a weighted blow-up along $C$ with weight $\s=(1,1,b,\ldots,b,0\ldots,0)$, where $b$ is a positive integer and the number of $b$'s is $ r $.
\end{Lemma}

\begin{proof}  By the assumptions and by  Theorem \ref{dimF}, we gain that $\dim F = r +1$; therefore $\dim C= n-r-2$.

First we prove that $p=f(F)$ is a smooth point of $Z$.  Take $n-r-2$  general functions $h_j \in H^0(X,\mathcal O_X)$, $j=1,\ldots,n-r-2$ and let  $X_j \subset X$ be the divisor defined by $h_j$. By Lemma \ref{vertical}, setting $X"=\bigcap_{j=1}^{n-r+2} X_j$, we have that $f" := f_{|X"}: X" \to Z"$ is a local contraction supported by  $K_{X"} + \tau L_{X"}$, it is birational, it contracts a divisor $F$ to the point $p=C \cap Z"$ and $\tau > r = \dim X" -2$. Note that $p$ is general in $C$  and hence $F$ is a general nontrivial fibre of $f$.  The contraction $f" : X" \to Z"$ satisfies the assumption of  Theorem \ref{n-2} and hence we have that  $f"$ is a weighted blow-up of a smooth point with weight $(1,1,b,\ldots,b)$, where $b$ is a positive integer. 

 Therefore we may assume that $Z" = f"(X")$ is smooth at $p$. Since $Z"$ is an intersection of Cartier divisors in $Z$, we conclude that $Z$ is smooth at $p$.

\medskip

The  proof now is by induction on the dimension of $F$, i.e. on $\dim F = r +1$; for this we   apply Lemma \ref{horizontal}. 

Assume $\dim F =1$. Since $X$ has  terminal singularities, which are in codimension $3$, $F$ is contained in the smooth locus of $X$, and hence, in the local set-up, we may assume that $X$ is smooth.  Therefore $X$ is a smooth blow-up (see for instance Corollary 4.11 in \cite{AW1}):  i.e. $f$ is the blow-up of $C=\{x_1=x_2=0  \} \subset Z$ with  weights $(1,1, 0 \dots, 0)$; in particular we have that 
$$
K_X= f^*K_Z+ E.
$$

If $\dim F= r +1 \ge 2$,  let $X'$ be a general element in $|L|$. By Theorem \ref{bpf} and Bertini's theorem, $X'$  has terminal $\Q$-factorial singularities. Consider the contraction $f':= f_{|X'}: X' \to Z'$; by Lemma \ref{horizontal}, it is a divisorial contraction supported by $K_{X'}+(\tau-1)L_{|X'}$ with fibres of dimension equal to $r$.  We have already proved that we may assume $Z=\C^n$ and that $Z'$ is smooth, therefore, by induction, we may assume  $C=\{x_1=\ldots=x_{r+2}=0  \} \subset Z'=\{x_{r+2}=0\}\subset Z'$ and that $f'$ is the smooth blow-up along $C$.

\medskip
Let $L_1$ be the Cartier divisor $f_*L$; we have  $L= f^*L_1 - bE$ for a positive integer $b$ and $bE$ is a Cartier divisor. 

Reasoning as in the proof of Theorem \ref{n-2}, by horizontal slicing  and by induction on $r$, we get the formulas
$$
K_X= f^*K_Z+ (rb+1)E \quad \mbox{ and } \quad  \tau= r+\frac{1}{b}.
$$

Since  $\m O_X(-bE)$ is  $f$-ample  we have
$$
X= \Proj \oplus_{d \ge 0}  f_*\m O_X(-dbE).
$$ 

\medskip
By Lemma \ref{push-forward} we have  to show that

$$
f_*\m O_X(-dbE)=I_{\s,d}=(x_1^{s_1}\cdots x_n^{s_n} : s_1 + s_2 + \sum_{j=3}^{r+2} s_j b \ge db).
$$ 
The proof is by double induction on $r \ge 0 $ and $d\ge 0$, and it is similar to the proof of Theorem \ref{n-2}.

Consider the exact sequence
$$
0 \to \m O_X(-L-dbE) \to \m O_{X} (-dbE)\to \m O_{X'}(-dbE) \to 0.
$$

Note that
$$
-L-dbE \sim_f 	K_X + (r+(d-1)+\frac{1}{b})L 
$$
and hence, by the Relative Kawamata-Viehweg Vanishing, we have

\begin{align}\label{sequence2}
0 \to  f_* \m O_X(-(d-1)bE) \stackrel{\cdot x_{r+2}}\rightarrow f_*\m O_{X} (-dbE)\to f_*\m O_{X'}(-dbE) \to 0.
\end{align}

By induction on $r$ we can assume that 

$$
f_*\m O_{X'}(-dbE)=(x_1^{s_1}\cdots x_{r+1}^{s_{r+1}}x_{r+3}^{s_{r+3}} \cdots x_{n}^{s_{n}} : s_1 +s_2 + \sum_{j=3}^{r+1} s_j b \ge db)
$$
where $s_j \in \bb N$. We have already treated above the case $r=0$.  By induction  on $d \ge 0$, we can assume that

$$
f_*\m O_{X}(-(d-1)bE)=(x_1^{s_1}\cdots x_n^{s_n} : s_1 + s_2 + \sum_{j=3}^{r+2} s_j b \ge (d-1)b),
$$
 the case $d=0$ being trivial.

Let $g=x_1^{s_1}\cdots x_n^{s_n} \in f_*\m O_{X}(-dbE)$.

If $s_{r+2} \ge 1$ then, looking at the sequence \eqref{sequence2}, $g$ comes from $f_*\m O_{X}(-(d-1)dE)$ by the multiplication by $x_{r+2}$  and so 
$$
s_1 + s_2 + \sum_{j=3}^{r+1} s_j b +s_{r+2}b\ge (d-1)b + s_{r+2} \ge db.
$$

Otherwise  $g \in f_*\m O_{X'}(-bdE)$ and it  satisfies 
$$
s_1 + s_2 + \sum_{j=3}^{r +1} s_j b= s_1 + s_2 + \sum_{j=3}^{r} s_j b \ge db.
$$

\end{proof}

\bigskip

\begin{proof}[Proof of Theorem \ref{main}]
First notice that, as in the first line of the proof of \ref{n-2}, $\dim C = n-r-2$.

We proceed as in  {\cite[Theorem 4.9]{KollarMori92}}.

By  Lemma \ref{local}, $C\cap \mathrm{Sing}(Z) \subsetneq C$; moreover since $Z$
has terminal singularities, $\mathrm{Sing}(Z)$  has codimension at least three. Therefore we can find a codimension three closed subset $S \subset Z$ such that $Z'=Z\backslash S$ is smooth and $f^{-1}(S)$ has codimension at least two in $X$. Moreover, by Lemma \ref{local}, we may assume that $X'=X \backslash f^{-1}(S)$ is a weighted blow-up of $Z'$ along $C'=C\backslash S$ with weight $\s=(1,1,b,\ldots,b,0,\ldots,0) \in \N^n$. This proves (i).

\medskip 

To prove (ii), since $X =  \Proj \bigoplus_{m\ge 0} f_* \m O_X(-mbE)$,  we need to show that
\begin{align}\label{eqsymbolic} 
f_* \m O_X(-mbE) = \m I^{(m)}.
\end{align}

Note that by Proposition \ref{projective-normality} we have
$$
f_*(\m O_X(-mbE))_{|Z'}= (\m I_{|Z'})^m.
$$

By definition of symbolic power, and using the fact proved in Lemma \ref{symbolic-powers}  
that $(\m I_{|Z'})^m=(\m I_{|Z'})^{(m)}=(\m I^{(m)})_{|Z'}$, we obtain

$$
i_*((\m I^{(m)})_{|Z'})=\m I^{(m)}.
$$

Therefore \ref{eqsymbolic} follows by
$$
i_*(f_*(\mathcal O_X(-mbE))_{|Z'})=f_*(\mathcal O_X(-mbE)),
$$
which is a consequence of the following general fact.

\begin{Lemma}\label{injection}
Let $f: U \to V$ be a proper morphism. Let $S \subset V$ be a closed subset such that the codimension of $f^{-1}(S)$ in $U$ is at least two. Let $\mathcal F$ be a sheaf that satisfies Serre's condition $S_2$ (e.g. U is normal and $\mathcal F$ is reflexive). Then $f_*\mathcal F= i_*(f_*\mathcal F _{|V \backslash S})$, where $i: V \backslash S \to V$ is the injection. 
\end{Lemma}

\bigskip

\end{proof}


\bigskip 



\begin{thebibliography}{99}

\bibitem{An1}
{\bibname M. Andreatta}.
\newblock  Some remarks on the study of good contractions,
\newblock  {\em Manuscripta Math.},
  volume 87, 1995,  359--367.
  
\bibitem{An11} 
{\bibname M. Andreatta}. 
\newblock Minimal Model Program with scaling and adjunction theory. 	
\newblock {\em Internat. J. M.} Vol. 24 (2013), no. 02. 

\bibitem{AW1}
{\bibname M. Andreatta \and  J. A. Wi\'sniewski}.
\newblock  A note on nonvanishing and applications,
\newblock  {\em Duke Math. J.},
volume 72 n. 3, 1993,  739--755.
  
\bibitem{AM69}
{\bibname M. F. Atiyah \and  I. G. Macdonald}. 
\newblock Introduction to commutative algebra, 
\newblock{ \em Addison-Wesley Publishing Co., Reading, Mass.-London-Don Mills}, 
Ont. 1969 ix+128 pp. 

\bibitem{BeltramettiSommese}
{\bibname M. Beltrametti \and A.J. Sommese}.
\newblock {\em The Adjunction Theory of Complex Projective Varieties}, 
volume 16 of {\em Expositions in Mathematics},
\newblock De Gruyter, Berlin-New York, 1995.
  
  
\bibitem{Hayakawa}
{\bibname T. Hayakawa}.
\newblock  Blowing ups of 3-dimensional terminal singularities, 
\newblock  Publ. Res. Inst. Math. Sci. 35 (1999), 515--570.
   
  
   \bibitem{Kaw}
{\bibname M. Kawakita}.
\newblock  Divisorial contractions in dimension three which contract divisors to smooth points,
\newblock  {\em  Invent. Math. }
vol. 145, no. 1, 2001, 105--119.
    
 \bibitem{KollarMori92}
{\bibname J. Koll{\'a}r \and S. Mori}. 
\newblock Classification of three-dimensional flips.
\newblock J. Amer. Math. Soc. 5 (1992), no. 3, 533–-703.
14E30 (14B07 14E05 14E35 14J30)     
    
\bibitem{KollarMori}
{\bibname J. Koll{\'a}r \and S. Mori}.
\newblock {\em Birational geometry of algebraic varieties}, volume 134 of {\em
  Cambridge Tracts in Mathematics}.
\newblock Cambridge University Press, Cambridge, 1998.
\newblock With the collaboration of C. H. Clemens and A. Corti, Translated from
  the 1998 Japanese original.
  
  

\bibitem{Laz}
{\bibname R. Lazarsfeld}.
\newblock Positivity in Algebraic Geometry {II}.
\newblock Ergebnisse der Mathematik und ihrer Grenzgebiete, vol. 49, Springer-Verlag, Berlin, 2004. 

  
\bibitem{Me97}
{\bibname M. Mella}.
\newblock Adjunction theory on terminal varieties.
\newblock  Complex analysis and geometry (Trento, 1995), 153–-164, Pitman Res. Notes Math. Ser., 366, Longman, Harlow, 1997.
      
\bibitem{Mo82}
{\bibname  S. Mori}.
\newblock Threefolds whose canonical bundles are not numerically effective,
\newblock  {\em Annals of Math},
  vol. 116, 1982, 133--176.
  
\bibitem{Mo88}
{\bibname S. Mori}.
\newblock Flip theorem and the existence of minimal models for $3$-folds,
\newblock  {\em Jour. AMS},
  vol. 1, 1998, 117--253.
  
  
\bibitem{Re87} {\bibname M. Reid}. Young person's guide to canonical singularities. Algebraic geometry, Bowdoin, 1985 (Brunswick, Maine, 1985), 345–-414, Proc. Sympos. Pure Math., 46, Part 1, Amer. Math. Soc., Providence, RI, 1987.  


\end{thebibliography}
\end{document}